\documentclass{amsart}
\usepackage[latin1]{inputenc}
\usepackage[active]{srcltx}
\usepackage{fullpage}
\usepackage{amsfonts,enumerate}
\usepackage[mathscr]{euscript}
\usepackage{epsfig}
\usepackage{latexsym}
\usepackage[all]{xy}
\usepackage{amssymb}
\usepackage{amsthm}
\usepackage{amsmath}
\usepackage{amsxtra}
\usepackage{xcolor}
\usepackage[colorlinks]{hyperref}
\usepackage{fancyvrb}
\hypersetup{citecolor=blue}

.tfm
.tfm

\theoremstyle{plain}
\newtheorem{prop}{Proposition}[section]
\newtheorem{thm}[prop]{Theorem}
\newtheorem{coro}[prop]{Corollary}

\newtheorem{lemma}[prop]{Lemma}
\newtheorem{conj}{Conjecture}

\newtheorem*{thm*}{Theorem}
\newtheorem*{lemma*}{Lemma}

\newtheorem*{prop*}{Proposition}

\theoremstyle{definition}

\theoremstyle{remark}

\newtheorem{remark}{Remark}

\numberwithin{table}{section}

\DeclareMathOperator{\Frob}{Frob}

\DeclareMathOperator{\Cl}{Cl}
\DeclareMathOperator{\Gal}{Gal}

\newcommand{\F}{\mathbb F}

\newcommand{\Om}{{\mathscr{O}}}

\newcommand{\Disc}{\Delta}

\newcommand{\GL}{{\rm GL}}
\newcommand{\SL}{{\rm SL}}

\newcommand{\E}{E_{(a,b,c)}}

\def\ZZ{\mathbb Z}

\def\<#1>{{\left\langle{#1}\right\rangle}}

\def\Z{{\mathbb Z}}             
\def\Q{{\mathbb Q}}             

\def\id#1{{\mathfrak{#1}}}      

\DeclareMathOperator{\trace}{{\mathrm{Tr}}}

\begin{document}
	
	\title{Asymptotic Fermat for signature $(4,2,p)$ over number fields}

	\author{Lucas Villagra Torcomian}
	\address{FAMAF-CIEM, Universidad Nacional de C\'ordoba. C.P: 5000,
		C\'ordoba, Argentina.}  \email{lucas.villagra@unc.edu.ar}

	\thanks{The author was supported by a CONICET grant.}
	
	
	\begin{abstract}	
	Let $K$ be a number field. Using the modular method,  we prove asymptotic results on solutions of the Diophantine equation $x^4-y^2=z^p$ over $K$, assuming some deep but standard conjectures of the Langlands programme when $K$ has at least one complex embedding. On the other hand, we give unconditional results in the case of totally real extensions having odd narrow class number and a unique prime above $2$. When modularity of elliptic curves over $K$ is known, for example when $K$ is real quadratic or the $r$-layer of the cyclotomic $\ZZ_2$-extension of $\Q$, effective asymptotic results hold.
	\end{abstract}
	\maketitle
	\section{Introduction}
\subsection{Historical background}
The study of Diophantine equations is one of the oldest problems in the area of number theory. Since Wiles' proof of Fermat's Last Theorem \cite{Wiles}, special attention has been paid in the so called Fermat-type equations, namely those of the form
\begin{equation}\label{eq:general}
Ax^q+By^r=Cz^p,
\end{equation}
where $A$, $B$ and $C$ are fixed non-zero pairwise coprime integers. In general, finding all solutions of a Diophantine equation is an extremely difficult problem, since there is no known algorithm to do that. On the other hand, Wiles' strategy used in his famous proof, referred to as the \textit{modular method}, can be used to prove the non-existence of (certain type of) solutions. The case $A=B=C=1$ is of special interest, where the modular method (together sometimes with other techniques) can be used to solve many of these equations. We refer to Tables 1 and 2 of \cite{IKOpp2} to see the known results until 2020. It is important to remark that the results collected in those tables are for integers solutions, i.e solutions with $x,y,z\in \Z$. The triple $(q,r,p)$ of~(\ref{eq:general}) is called the \textit{signature} of the equation and a \textit{primitive} solution $(x,y,z)=(a,b,c)$ is a solution where $a$, $b$ and $c$ are pairwise coprime. We also say that $(a,b,c)$ is \textit{non-trivial} if $abc\neq0$.

In \cite{Darmon1995}, Darmon and Granville proved that given a number field $K$ and a signature $(q,r,p)$ such that $1/q+1/r+1/p<1$, there are finitely many primitive solutions $(a,b,c)$ of~(\ref{eq:general}) with $a,b,c\in \Om_K$, where $\Om_K$ is the ring of integers of $K$. Although their result works for number fields, equation~(\ref{eq:general}) has been mainly studied for $K=\Q$. One of the main reasons is that modularity of elliptic curves plays a crucial role in the modular method, so the lack of results for modularity of elliptic curves over number fields is a big obstruction. Nevertheless, during the last years there has been lot of progress in that direction, and Diophantine results came to light. In 2015, using modularity of elliptic curves over real quadratic fields \cite{FHS}, Freitas and Siksek \cite{FS} shown that asymptotic Fermat for signature $(p,p,p)$ holds for infinitely many real quadratic fields $K$. That is, there exists a bound $B_K$ (depending only on $K$) such that if $p>B_K$, then the only solutions $(a,b,c)\in\Om_K^3$ of $x^p+y^p=z^p$ are the trivial ones. In 2018, \c{S}eng\"un and Siksek \cite{SS} gave sufficient conditions in order to asymptotic Fermat for signature $(p,p,p)$ holds for a given number field $K$. In that case, the authors assumed deep conjectures of then Langlands programme, such as modularity of elliptic curves over $K$. For similar results of equation~(\ref{eq:general}) with different values of $A$, $B$ and $C$ and signature $(p,p,p)$ we refer the reader to \cite{Deconinck, KaraOzman}. Following the same methodology, in 2020 I\c{s}ik, Kara and \"Ozman \cite{IKOpp2} worked with signature $(p,p,2)$ over totally real fields. In 2022, they also obtained asymptotic results for signature $(p,p,3)$ over arbitrary number fields (assuming same conjectures than \cite{SS}, see \cite{isik_kara_ozman_2022}); in the same year Mocanu \cite{mocanu} made improvements for both signatures assuming $K$ to be totally real.

In the preset article, we are going to work with signature $(4,2,p)$, where $p$ is an odd prime. Here we have a brief summary of the background for this case. When $A=C=1$, we have the equation
\begin{equation}\label{eq:B42p}
x^4+By^2=z^p.
\end{equation}
In this case the first asymptotic result was obtained by Ellenberg \cite{Ellenberg} for $B=1$ and $K=\Q$, proving that there are no non-trivial primitive solutions if $p\ge 211$. Then, a similar result over $\Q$ was proved for $B=2,3$ and $p\ge 349, 131$, respectively, by Dieulefait and Jim\'enez--Urroz \cite{DieulefaitJimenez}. In \cite{PacettiVillagra} and \cite{PacettiVillagra2}, Pacetti and the author gave a recipe (using the modular method) to get asymptotic results for $K=\Q$ and $B$ an arbitrary integer, and explicit bounds for $p$ were given when $B=5,6,7$. Nevertheless, one question still unanswered in general is the following:  given $B$, is there a bound $B'$ such that there are no non-trivial primitive solutions of~(\ref{eq:B42p}) for $p>B'$?. 

Note that a solution $(a,b,c)\in \ZZ^3$ of~(\ref{eq:B42p}) induces a solution $(a,\sqrt{-B}b,c)\in\Om_K^3$ of the equation
\begin{equation}\label{eq:42p}
x^4-y^2= z^p,
\end{equation}
where $K=\Q(\sqrt{-B})$ (the reader can also note that equation~(\ref{eq:42p}) is the same as (\ref{eq:general}) with $A=B=C=1$ and signature $(4,p,2)$, if $p$ is odd). Then, instead of studying equation~(\ref{eq:B42p}) over $\Q$ for different values of $B$, it would be enough to understand solutions of~(\ref{eq:42p}) over quadratic fields, or more generally over arbitrary number fields. This is the direction in which we focus. The goal of the present article is to give asymptotic results for equation~(\ref{eq:42p}) over a number field $K$.

\subsection{Main results}

Let $K$ be a number field and denote by $h_K^+$ its narrow class number. Let \[S_K=\{\id{P}: \id{P} \text{ is a prime of } K \text{ above } 2\}.\]
We denote by $T_K$ the subset of $S_K$ defined by 
\[T_K:=\left\{\id{P}\in S_K:|v_\id{P}(\alpha/\beta)|\le 6v_{\id{P}}(2) \text{ for every } (\alpha,\beta,\gamma)\in\Om_{S_K}^\times\times\Om_{S_K}^\times\times\Om_{S_K} \text{ satisfying } \alpha+\beta=\gamma^2\right\}. \] 
 The following are two important hypothesis that we will use:
\[\textbf{(}\textbf{H}_\textbf{1}\textbf{)}:	2\nmid h_K^+ \ \text{ and } \ \# S_K=1, \qquad \textbf{(}\textbf{H}_\textbf{2}\textbf{)}:C_{S_K}(K)[2]=\{1\} \  \text{ and } \ T_K\neq\emptyset. \]
We follow the notation used in \cite{mocanu}, where  $\Cl_S(K)$ means $\Cl(K)/\langle [\id{P}] \rangle_{\id{P}\in S}$ and $\Cl_S(K)[n]$ denotes the set of its $n$-torsion elements. When all $\id{P}\in S$ are principal, $\Cl_S(K)$ is just the class group $\Cl(K)$. In such a case, if $p$ is a prime number then the condition of $\Cl_S(K)[p]$ being trivial is just that $p\nmid h_K$. Moreover, if $2\nmid h_K^+$ then $\Cl_{S_K}(K)[p]=\{1\}$. We can now formulate our results.
	\begin{thm}\label{thm:main}
		Let $K$ be a number field satisfying \textbf{(}\textbf{H}$_{\textbf{1}}$\textbf{)} or \textbf{(}\textbf{H}$_\textbf{2}$\textbf{)}. If $K$ has at least one complex embedding, assume furthermore Conjectures~\ref{conj:modularity} and~\ref{conj:fakecurves}. Let $\id{P}$ the only prime in $S_K$ in case that \textbf{(}\textbf{H}$_\textbf{1}$\textbf{)} holds or be an element of $T_K$ in case that \textbf{(}\textbf{H}$_\textbf{2}$\textbf{)} holds. Then, there exists a bound $B_{K}$ (depending only on $K$) such that if $p>B_K$  the equation $x^4-y^2=z^p$ does not have non-trivial primitive  solutions $(a,b,c)\in\Om_K^3$ with $\id{P}\mid c$.
	\end{thm}

\begin{remark}
	The condition  $\id{P}\mid c$ for a putative solution $(a,b,c)$ is necessary. Otherwise Theorem~\ref{thm:main} would be false, since $K=\Q(\sqrt{2})$ satisfies \textbf{(}\textbf{H}$_\textbf{1}$\textbf{)} but $(\pm1,\pm \sqrt{2},-1)$ is a non-trivial primitive solution for every $p>2$. The solution $(\pm1,\pm \sqrt{2},-1)$ is non-trivial according to our definition, but it is special in the sense that it is solution of~(\ref{eq:42p}) for every odd prime exponent $p$. It is clear that this phenomenon is possible just if $c\in\{0,\pm1\}$. Note that there are also non-trivial primitive solutions with $c\neq1$ odd and $p$ relatively large (i.e. $1/4+1/2+1/p<1$). For instance over $K=\Q$ we have
	\[11^4-122^2=(-3)^5.\]
\end{remark}


We say that the constant $B_K$ it is \textit{effective} when is effectively computable. We have important consequences when we restrict our attention to particular families of totally real number fields, such as some quadratic extensions and the $r$-th layers of the cyclotomic $\Z_2$-extension $\Q_{r,2}:=\Q(\zeta_{2^{r+2}})^+$. The effectivity of the bound in these cases will follow from modularity of elliptic curves over real quadratic fields \cite{FHS} for the first family, and for the extensions $\Q_{r,2}$ by the more recent work of Thorne \cite{Thorne} in which  modularity of elliptic curves over $\Q_\infty$ is proven.

\begin{coro}\label{coro:quadraticase}
	Let $d>0$  be a rational prime satisfying $d \equiv 5 \pmod 8$.  Then, there is an effective constant $B_d$  such that
	for any rational prime $p > B_d$, the equation a
	$x^4-y^2=z^p$ does not have
non-trivial primitive
	solutions $(a, b, c) \in \Om_{\Q(\sqrt{d})}^3$ with $2\mid c$.
\end{coro}

\begin{coro}\label{thm:Q2} Let $\id{P}$ the only prime of $\Q_{r,2}$ above $2$. Then, there exists an effective bound $B_{r}$ such that if $p>B_{r}$ then the equation $x^4-y^2=z^p$ does not have non-trivial primitive solutions $(a,b,c)\in\Om_{\Q_{r,2}}^3$ such that $\id{P}\mid c$.
\end{coro}

\subsection{The strategy} Let $K$ be a number field and let $\Gal_K:=\Gal(\overline{\Q}/K)$ its absolute Galois group. If $E$ is an elliptic curve over $K$ and $p$ a rational prime we denote by  
\[\rho_{E,p}:\Gal_K\to\GL_2(\ZZ_p)\]
the Galois representation obtained from the action of $\Gal_K$ on the $p$-adic Tate module $T_p(E)$. We denote by $\overline{\rho}_{E,p}$  its mod $p$ reduction. Let $\id{P}$ be a fixed prime in $K$ dividing $2$, i.e,  $\id{P}\in S_K$. The strategy of the modular method applied to this case can be summarized in the following steps.

\begin{enumerate}[(i)]
	\item \textbf{Frey curve:}  To a putative non-trivial primitive solution $(a,b,c)\in\Om_K^3$  of equation~(\ref{eq:42p}) with $\id{P}\mid c$ we will attach a Frey elliptic curve $\E$ over $K$ with a $K$-rational point of order $2$, semistable reduction outside $S_K$ and potentially multiplicative reduction at $\id{P}$.
	
	\item \textbf{Modularity and lowering the level:}  Assuming some hypothesis and results for elliptic curves over number fields and Galois representations, we will get an authomorphic form $\id{f}$ over $K$ which lies in a space that does not depend on the solution $(a,b,c)$ neither on $p$ such that 
	\[ \label{eq:sim} \tag{$\star$}
		\overline{\rho}_{\E,p}\sim\overline{\rho}_{\id{f},\varpi},  \]
	where $\varpi$ is some prime of $\Q_\id{f}$ dividing $p$.
	
	\item \textbf{Eichler-Shimura:} The idea of this step is to prove (assuming $p$ large enough) the existence of an elliptic curve $E_\id{f}$ over $K$ arising from $\id{f}$. Then, using~(\ref{eq:sim}) we can prove that $E_\id{f}$ has a $K$-rational point of order $2$, good reduction outside $S_K$ and potentially multiplicative reduction at $\id{P}$.
	\item \textbf{Contradiction:} Using hypothesis \textbf{(}\textbf{H}$_\textbf{1}$\textbf{)} or \textbf{(}\textbf{H}$_\textbf{2}$\textbf{)}, we can prove that there are no elliptic curves over $K$ with a $K$-rational point of order $2$, good reduction outside $S_K$ and potentially multiplicative reduction at $\id{P}$, getting a contradiction.
\end{enumerate}

\begin{remark}\label{rem:signature}
	The reason that makes signatures $(p,p,p)$, $(p,p,2)$ and $(p,p,3)$ special is that in every case there exists a rational prime $\ell$ ($\ell=2,2,3$ respectively) and a Frey elliptic curve $E$ over $K$ (attached to a certain type of solution over $\Om_{K}^3$) such that $E$ has potentially multiplicative reduction in a fixed prime above $\ell$, semistable reduction outside $\ell$ and a $K$-rational point of order $\ell$. To the best of the author's knowledge, signature $(4,2,p)$  is the only remaining case with such properties.
\end{remark}

The next section is devoted to the study of equation~(\ref{eq:42p}), following the above strategy. To make the aforementioned steps more accessible, each subsection is rendered as self-contained as possible, repeating the relevant material, specially from \cite{FS,SS}. At the end of the section we will  prove the main results, and we will see how the same approach can be used to get new asymptotic results for signature $(p,p,2)$.

\subsection*{Acknowledgments}
I am very grateful to Ariel Pacetti for his support, many useful conversations and for helpful comments to improve this article. I also thank Diana Mocanu for the useful discussions and for the reading of this paper. Finally, I would like to thank Nuno Freitas, for the careful reading and helpful remarks.

\section{The equation $x^4-y^2=z^p$}\label{section:equation}

\subsection{Frey curve and related facts}
	For $(a,b,c)\in\Om_K^3$ solution of~(\ref{eq:42p}) we associate the Frey elliptic curve
	\begin{equation}\label{eq:freycurve}
\E:y^2=x^3+4ax^2+2(a^2+b)x,
	\end{equation}
	defined over $K$. Note that $(0,0)$ is a $K$-rational point of order $2$. Its discriminant equals $\Disc_{\E} = 2^9(a^2+b)c^p$ and its
	$j$-invariant equals
	$j_{\E}=\frac{2^6(5a^2-3b)^3}{ c^p(a^2+b)}$.  We denote by $\id{N}_{\E}$ its conductor. As was mentioned in the introduction, we only want to consider primitive solutions of~(\ref{eq:42p}), since we know that there are finitely many. Lemma~\ref{lemma:infinitely} below proves that there are infinitely many solutions removing this hypothesis. Then, from now on we will assume $(a,b,c)$ to be primitive.
	
	\begin{lemma}\label{lemma:infinitely}
		Let $p>3$ be a rational prime. Then equation~(\ref{eq:42p}) has infinitely non-primitive solutions.
	\end{lemma}
\begin{proof}
	Let $u,v\in\Om_K$ arbitrary elements such that $r=u^4-v^2$ is not equal to $\pm 1$. If $p\equiv1\pmod 4$ then $(ur^{\frac{p-1}{4}},vr^{\frac{p-1}{2}},r)$ is solution of~(\ref{eq:42p}). If $p\equiv3\pmod4$ then $(ur^{\frac{3p-1}{4}},vr^{\frac{3p-1}{2}},r^3)$ is solution of~(\ref{eq:42p}).
\end{proof}
	\begin{remark}\label{rem:valuation2}
		Fix $\id{P}\in S_K$. Let $(a,b,c)\in\Om_K^3$ a primitive solution of~(\ref{eq:42p}), where $p>2v_\id{P}(2)$ and $\id{P}\mid c$. Since $(a^2+b)(a^2-b)=c^p$ and $a^2+b=(a^2-b)+2b$ then $\id{P}\mid \gcd(a^2+b,a^2-b)$. Moreover, if  $v_\id{P}(a^2-b)>v_\id{P}(2)$ then $v_\id{P}(a^2+b)=v_\id{P}(2)$, since $\id{P}\nmid b$. On the other hand, if $v_\id{P}(a^2-b)\le v_\id{P}(2)$ we have that  $v_\id{P}(a^2+b)\ge v_\id{P}(a^2-b)$. Thus,  $v_\id{P}(a^2+b) > v_\id{P}(2)$, since otherwise we would have
		\[p\le pv_\id{P}(c)\le 2v_\id{P}(a^2+b) \le 2v_\id{P}(2).\]
		Hence, $\min \{v_\id{P}(a^2+b), v_\id{P}(a^2-b)\} > v_\id{P}(2)$. Then, interchanging $b$ by $-b$ if necessary, we can assume $v_\id{P}(a^2+b)=v_\id{P}(2)$ and $v_\id{P}(a^2-b)=  pv_\id{P}(c) - v_\id{P}(2)$. 
	\end{remark}
	
	
	Let $K$ be a number field, $E/K$ an elliptic curve of conductor $\id{N}_E$ and $p$ a rational prime. Define
	\[\id{M}_p=\prod_{\substack{\id{q}\mid\mid\id{N}_E\\ p\mid v_\id{q}(\Delta_\id{q})}}\id{q}, \qquad \id{N}_p=\frac{\id{N}_E}{\id{M}_p},\]
	where $\Delta_\id{q}$ is the discriminant of a local minimal model of $E$ at $\id{q}$. 
	\begin{lemma}\label{lemma:condutor}
		For all primes $\id{q}\notin S_K$, the model of $\E$ given in~(\ref{eq:freycurve}) is minimal, semistable and satisfies $p\mid v_\id{q}(\Delta_{\E})$. The representation  $\overline{\rho}_{\E,p}$ is odd, its determinant is the cyclotomic character and is finite flat at every prime $\id{p}$ of $K$ dividing $p$. Moreover,
		\begin{equation*}
			\id{N}_{\E}=\prod_{\id{P}\in S_K}\id{P}^{r_\id{P}}\prod_{\substack{\id{q}\mid c, \\ \id{q}\notin S_K}}\id{q}, \qquad \id{N}_p=\prod_{\id{P}\in S_K}\id{P}^{r'_\id{P}},
		\end{equation*}
	where $0 \le r'_\id{P} \le r_\id{P} \le 	2+6v_\id{P}(2)$.
	\end{lemma}
\begin{proof}
	Let $(a,b,c)$ a primitive solution of~(\ref{eq:42p}). From the discriminant formula, clearly
	$v_{\id{q}}(\Disc_{\E}) = pv_{\id{q}}(c) +
	v_{\id{q}}(a^2+b)$.  If
	$\id{q} \mid \gcd(a^2+b,a^2-b)$ then
	$\id{q} \mid 2$ (because $(a,b,c)$ is primitive), hence the equality
	$c^p = a^4-b^2=(a^2+b)(a^2-b)$ implies $p\mid v_\id{q}(a^2+b)$ for all $\id{q}\notin S_K$. Moreover,  $\gcd(5a^2-3b,c)$ is supported on $S_K$. Thus, if $\id{q}\not\in S_K$ is a prime dividing $\Delta_{\E}$ then $v_\id{P}(j_{\E})<0$. Hence, the given model is minimal at $\id{q}$ and $\E$ is semistable away from $S_K$. Furthermore, \[v_\id{q}(\Delta_\id{q})=v_\id{q}(\Delta_{\E})=pv_\id{q}(c)+v_\id{q}(a^2+b),\] so $v_\id{q}(\Delta_\id{q})$ is divisible by $p$. It follows from \cite{Serre} that $\overline{\rho}_{\E,p}$ is unramified at $\id{q}$ if $\id{q}\nmid p$ and is finite flat at every $\id{p}$ dividing $p$. On the other hand, if $\id{P}\in S_K$ then, by \cite[Theorem IV.10.4]{SilvermanAdv}, $r_\id{P}=v_\id{P}(\id{N}_{\E})\le 2+6v_\id{P}(2)$. The formula for $\id{N}_p$ follows easily from its definition. The statement about the determinant is a well-known consequence of the
	theory of the Weil pairing on $\E[p]$. This immediately implies oddness.
\end{proof}

\begin{remark}\label{rem:serreconductor}
	Note that the Serre conductor  of $\overline{\rho}_{\E,p}$ (prime-to-$p$ part of its Artin conductor) is supported on $S_K$, and belongs to a finite set depending only on $K$.
\end{remark}
To ease notation, and considering Remark~\ref{rem:valuation2},  we define for $\id{P}\in S_K$ the set \[W_\id{P}=\{(a,b,c): (a,b,c)\in\Om_K^3 \text{ is a non-trivial primitive solution of } (\ref{eq:42p}) \text{ such that } v_\id{P}(a^2+b)=v_\id{P}(2)\}.\]

\begin{lemma}\label{lemma:2ismultiplicative} Let $\id{P}$ be a prime of $K$ above $2$ and let $p>8v_\id{P}(2)$. If $(a,b,c)\in W_\id{P}$ then $\E$ has potentially multiplicative reduction at $\id{P}$ and $p\mid \#\overline{\rho}_{\E,p}(I_{\id{P}})$. 
\end{lemma}
\begin{proof}
	Since $\id{P}\nmid ab$ and $5a^2-3b=5(a^2+b)-8b$ then we have $v_{\id{P}}(5a^2-3b)=v_{\id{P}}(2)$.  The statement about the Galois representation follows from \cite[Lemma 5.1]{SS}  and the fact that
	\[v_{\id{P}}(j_{\E})=8v_{\id{P}}(2)-pv_{\id{P}}(c).\]
\end{proof}

\subsection{Modularity}
If $E$ is an elliptic curve over a number field $K$, we say that $E$ is modular if there is an isomorphism of compatible systems of Galois representations 
\begin{equation}\label{eq:form}
	\rho_{E,p}\sim\rho_{\id{f},\varpi},
\end{equation}
where  $\id{f}$ is an authomorphic form over $K$ of weight $2$ with coefficient field $\Q_\id{f}$ and $\varpi$ is a prime in $\Q_\id{f}$ above $p$. 
\subsubsection{Totally real case}
In the totally real case, $\id{f}$ comes from a Hilbert eigenform of parallel weight $2$ over $K$. We have the following result due to  Freitas, Le Hung and Siksek.
\begin{thm}[\cite{FHS}, Theorems 1 and 5]\label{thm:FHS}
	Let $K$ be a totally real field. There are at most finitely many $\overline{K}$-isomorphism classes of non-modular elliptic curves $E$ over $K$. Moreover, if $K$	is real quadratic, then all elliptic curves over $K$ are modular.
\end{thm}

\begin{coro}\label{coro:modularity}
	Let $K$ be a totally real field and let $\id{P}$ a fixed prime in $S_K$. There is some constant $A_K$ such that if $p>A_K$ and $(a,b,c)\in W_{\id{P}}$ then $\E/K$ is modular.
\end{coro}

\begin{proof}
	By Theorem~\ref{thm:FHS}, there are at most finitely many possible $\overline{K}$-isomorphism
	classes of elliptic curves over $E$ which are not modular. Let $j_1, j_2, \cdots , j_n \in K$
	be the $j$-invariants of these classes. Since $(a,b,c)$ is non-trivial we can define $\lambda:= \frac{2^2(a^2-b)}{a^2+b}$, a non-zero element. The $j$-invariant of $\E$ is
	\[j(\lambda)=2^8(\lambda+1)^3\lambda^{-1}.\]
	Each equation $j(\lambda)=j_i$ has at most three solutions $\lambda\in K$. Thus there are values $\lambda_1,\cdots,\lambda_m$ (where $m\le 3n$) such that if $\lambda\neq\lambda_k$ for al $k$ then $\E$ is modular. If $\lambda=\lambda_k$ for some $k$, then using  $v_\id{P}(a^2+b)=v_\id{P}(2)$ we have 
	\[v_{\id{P}}(\lambda_k)=pv_{\id{P}}(c)\ge p.\]
	Finally, taking $A_K:=\max\{v_\id{P}(\lambda_k)\}_{k=1}^m$ we get the desired bound.
\end{proof}
\begin{remark}
	In order to prove the finiteness of $\overline{K}$-isomorphism classes of elliptic curves $E$ which are not modular, in \cite{FHS} the authors proved that such a curve $E$ gives rise to a $K$-rational point on a modular curve of genus $g>1$. Then by Falting's theorem \cite{Faltings1983} there are finitely many curves $E$. Since Falting's theorem is ineffective so is the bound $A_K$.
\end{remark}

\subsubsection{General case}
When $K$ is an arbitrary number field (admitting possibly a complex embedding), the $\id{f}$ of~(\ref{eq:form}) will be a (weight 2) complex eigenform over $K$. We refer to \cite[Section 2]{SS} to see all the concepts regarding complex eigenforms and mod $p$ eigenforms over $K$. Since there are no modularity results in general, we have to make use of the following conjecture.
\begin{conj}[\cite{FKS}, Conjecture 4.1]\label{conj:modularity}
	Let $\overline{\rho}:\Gal_K\to\GL_2(\overline{\F}_p)$ be an odd, irreducible, continuous representation with Serre conductor $\id{N}$ and
	such that $\det(\overline{\rho}) =\chi_p$ is the mod $p$ cyclotomic character. Assume that $p$ is unramified in $K$
	and that $\overline{\rho}|_{\Gal_{K_{\id{p}}}}$
	arises from a finite-flat group scheme over $\Om_{K_\id{p}}$ 
	for every prime $\id{p}\mid p$. Then
	there is a weight two, mod $p$ eigenform $\theta$ over $K$ of level $\mathfrak{N}$ such that for all primes $\id{q}$ coprime
	to $p\id{N}$, we have
	\[\trace(\overline{\rho}(\Frob_\id{q}))=\theta(T_\id{q}),\]
	where $T_{\id{q}}$ denotes de Hecke operator at $\id{q}$.
\end{conj}

\subsection{Level reduction and Eichler-Shimura}

In order to prove the surjectivity of our residual Galois representation we use the following well-known result about subgroups of $\GL_2(\F_p)$ (see \cite[Lemma 2]{SD} and \cite[Propositions 2.3 and 2.6]{Etropolski}).
\begin{prop}\label{prop:repimage}
	Let $E$ be an elliptic curve over a number field $K$ of degree $d$ and let $G \le
	\GL_2 (\F_p)$ be the image of the mod $p$ Galois representation of $E$. Then the following holds:
	\begin{itemize}
		\item if $p\mid \# G$ then either $G$ is reducible or $G$ contains $\SL_2(\F_p)$ and hence is absolutely irreducible.
		\item if $p\nmid \# G$ and $p>15d+1$ then $G$ is contained in a Cartan subgroup or $G$ is contained
		in the normalizer of Cartan subgroup but not the Cartan subgroup itself.
	\end{itemize}
\end{prop}

\begin{prop}[\cite{SS}, Proposition 6.1]\label{prop:irreducible}
	Let $L$ be a Galois number field and let $\id{q}$ be a prime of $L$. There
	is a constant $B_{L,\id{q}}$ such that the following is true: let $p > B_{L,\id{q}}$ be a rational prime and let $E/L$ be an elliptic curve that is semistable at all $\id{p} \mid p$, having potentially
	multiplicative reduction at $\id{q}$. Then $\overline{\rho}_{E,p}$ is irreducible.
\end{prop}

\begin{lemma}\label{lemma:surjectivity}
	Let $K$ be a number field  and let $\id{P}\in S_K$. There is a constant $C_K$ such that if $p>C_K$ and $(a,b,c)\in W_\id{P}$ then $\overline{\rho}_{\E,p}$ is surjective.
\end{lemma}
\begin{proof}
	The proof follows the same argument given in  \cite[Corollary 6.2]{SS}. By Lemma~\ref{lemma:2ismultiplicative}, if $p>8v_\id{P}(2)$, we have that $\E$ has potentially multiplicative reduction at $\id{P}$ and, by Lemma~\ref{lemma:condutor}, it is 
	semistable away from $S_K$. Let $L$ be a Galois closure of $K$, and let $\id{q}$ be a prime of $L$ above $\id{P}$. Now, applying Proposition~\ref{prop:irreducible} we get a constant $C_{L,\id{q}}$ such
	that $\overline{\rho}_{\E,p}(\Gal_L)$  is irreducible whenever $p > C_{L,\id{q}}$. Since there are only finitely many primes $\id{q}$ in $L$ dividing $\id{P}$, and $L$ only depends on $K$ then at the end we can obtain a bound depending only on $K$ and we denote it by $C_K$ (we assume $C_K>8v_\id{P}(2)$, since we applied Lemma~\ref{lemma:2ismultiplicative}). Now, applying Lemma~\ref{lemma:2ismultiplicative} again, we see that the image of $\overline{\rho}_{\E,p}$ contains an element of order $p$. By Proposition~\ref{prop:repimage} any subgroup of $\GL_2 (\F_p)$ having an element of order $p$ is either
	reducible or contains $\SL_2 (\F_p)$. As $p > C_K > 8v_\id{p}(2)$, the image contains $\SL_2 (\F_p)$.
	Finally, again taking $C_K$ large enough if needed, we can ensure that $K\cap \Q(\zeta_p) =\Q$. Therefore,
	$\chi_p = \det(\overline{\rho}_{\E,p})$ is surjective, giving the following short exact sequence
	\[1\xrightarrow{}\SL_2(\F_p)\xrightarrow{}\overline{\rho}_{\E,p}(\Gal_K)\xrightarrow{\det}\F_p^\times\xrightarrow{}1\]
and completing the proof.
\end{proof}

\subsubsection{Totally real case}

\begin{thm}[\cite{FS}, Theorem 7]\label{thm:FS}
	Let $K$ be a totally real field and $E/K$ an elliptic curve of conductor $\id{N}$. Let $p$ be a rational prime. For a prime $\id{q}$ of $K$, let $\Delta_\id{q}$ and $\id{N}_p$ as in the previous subsection. Suppose the following statements hold:
	\begin{itemize}
		\item $p\ge 5$, the ramification index $e(\id{p}\mid p)<p-1$, and $\Q(\zeta_p)^+\not\subset K$,
		\item $E$ is modular,
		\item $\overline{\rho}_{E,p}$ is irreducible,
		\item $E$ is semistable at all $\id{p}\mid p$,
		\item $p\mid v_\id{p}(\Delta_\id{q})$ for all $\id{p}\mid p$.
	\end{itemize}
	Then, there is a Hilbert eigenform $\id{f}$ of parallel weight $2$ that is new at level $\id{N}_p$ and some prime $\varpi$ of $\Q_\id{f}$ above $p$ such that $\overline{\rho}_{E,p}\sim\overline{\rho}_{\id{f},\varpi}$.
\end{thm}

Assuming $K$ totally real, and following the works of Blasius \cite{Blasius}, Darmon \cite{DarmonAms} and Zhang \cite{Zhang}, Freitas and Siksek obtained the following result.

\begin{thm}[\cite{FS}, Corollary 2.2]\label{thm:existencecurve}
	Let $E$ be an elliptic curve over a totally real field $K$, and $p$ be and odd prime. Suppose that $\overline{\rho}_{E,p}$ is irreducible, and $\overline{\rho}_{E,p}\sim\overline{\rho}_{\mathfrak{f},\varpi}$ for some Hilbert newform $\mathfrak{f}$ over $K$ of level $\id{N}$ and parallel weight $2$ which satisfies $\Q_{\mathfrak{f}}=\Q$. Let $\id{q}\nmid p$ be a prime ideal of $\Om_K$ such that
	\begin{itemize}
		\item $E$ has potentially multiplicative reduction at $\id{q}$,
		\item $p\mid\#\overline{\rho}_{E,p}(I_\id{q})$,
		\item $p\nmid(\text{Norm}_{K/\Q}(\id{q})\pm1)$.
	\end{itemize}
	Then there is an elliptic curve $E_\mathfrak{f}/K$ of conductor $\id{N}$ with the same $L$-function as $\mathfrak{f}$.
\end{thm}

Let $K$ be a totally real and $\id{P}\in S_K$ a fixed prime. If $(a,b,c)\in W_\id{P}$ with $p$ large enough, then by Corollary~\ref{coro:modularity}  $\E$ is modular and $\overline{\rho}_{\E,p}$ is irreducible, by Lemma~\ref{lemma:surjectivity}. Applying Lemma~\ref{lemma:condutor} and Theorem~\ref{thm:FS} we have that $\overline{\rho}_{\E,p}\sim\overline{\rho}_{\id{f},\varpi}$ for a Hilbert newform $\id{f}$ of level $\id{N}_p$ and some prime $\varpi\mid p$ of $\Q_\id{f}$. Enlarging $p$ if necessary we can assume $\Q_\id{f}=\Q$ (see \cite[Lemma 7.2]{SS}). Then,  Theorem~\ref{thm:existencecurve} gives the existence of a curve $E_\id{f}$ with conductor $\id{N}_p$ such that $\overline{\rho}_{\E,p}\sim\overline{\rho}_{E_\id{f},p}$.

\subsubsection{General case}

In general if $K$ admits a complex embedding, we assume Conjectures~\ref{conj:modularity} and~\ref{conj:fakecurves} and we apply the following proposition to prove Lemma~\ref{lemma:eichlershimura}.

\begin{prop}[\cite{SS}, Proposition 2.1]\label{prop:bajadadenivel}
	There is an integer $B(\id{N})$, depending only on $\id{N}$, such
	that for any prime $p > B(\id{N})$, every weight two mod $p$ eigenform of level $\id{N}$ lifts to a complex
	one.
\end{prop}

\begin{conj}[\cite{SS}, Conjecture 4.1]\label{conj:fakecurves}
	Let $\id{f}$ be a weight $2$ complex eigenform over $K$ of level $\id{N}$
	that is non-trivial and new. If $K$ has some real place, then there exists an elliptic curve $E_\id{f}/K$
	of conductor $\id{N}$ such that
	\begin{equation}\label{eq:ellipticcurve}
		\# E_\id{f}(\Om_K/\id{q})=1+\text{ Norm}_{K/\Q}(\id{q})-\id{f}(T_\id{q}) \text{ for all } \id{q}\nmid\id{N}.
	\end{equation}
	
	If $K$ is totally complex, then there exists either an elliptic curve $E_\id{f}$ of conductor $\id{N}$ satisfying~(\ref{eq:ellipticcurve}) or a fake elliptic curve $A_\id{f}$, of conductor $\id{N}^2$, such that
	\begin{equation}\label{eq:fakeelliptic}
		\# A_\id{f}(\Om_K/\id{q})=(1+\text{ Norm}_{K/\Q}(\id{q})-\id{f}(T_\id{q}))^2 \text{ for all } \id{q}\nmid\id{N}.
	\end{equation}
	
\end{conj}

\begin{lemma}\label{lemma:eichlershimura}
	There is a non-trivial new, weight $2$ complex eigenform $\mathfrak{f}$ which has an associated elliptic curve $E_{\mathfrak{f}}/K$ of conductor $\mathfrak{N}'$ dividing $\id{N}_p$.
\end{lemma}
\begin{proof}
	By Lemma~\ref{lemma:surjectivity},  $\overline{\rho}_{\E,p}$ is surjective hence absolutely irreducible for $p> C_K$. Now, by Conjecture~\ref{conj:modularity} there is a weight $2 \mod p$ eigenform $\theta$ over $K$ of level $\mathfrak{N}_p$ such that for all primes $\id{q}$ coprime to $p\mathfrak{N}_p$ we have
	\[\trace(\overline{\rho}_{\E,p}(\Frob_\id{q}))=\theta(T_\id{q}).\]
	Since there are only finitely many possible values for $\mathfrak{N}_p$ (depending only on $K$, see Remark~\ref{rem:serreconductor}) taking $p$ large enough we have that for any level $\mathfrak{N}_p$ there will be  a weight $2$ complex eigenform $\mathfrak{f}$ with level $\mathfrak{N}_p$ that is a lift of $\theta$, by Proposition~\ref{prop:bajadadenivel}.
	Recall that we can assume $\Q_\mathfrak{f}=\Q$ for $p$ large enough. Note that $\id{f}$ is non-trivial, since $\overline{\rho}_{\E,p}$ is irreducible. If $\id{f}$ is not new, we can
	replace it with an equivalent new eigenform that is of smaller level. Then, we can obtain $\id{f}$ new and with level $\id{N}'$ dividing $\id{N}_p$.
	Applying Conjecture~\ref{conj:fakecurves} we get that $\id{f}$  either has an associated elliptic curve $E_\id{f}/K$ of conductor $\id{N}'$, or has an associated fake elliptic curve $A_\id{f}/K$ of conductor $\id{N}'^2$. The fake elliptic curve case can be discarded by \cite[Lemma 7.3]{SS}.
\end{proof}

In either case ($K$ totally real or not), we denote by $E'$ the elliptic curve $E_{\mathfrak{f}}$. We have that $\overline{\rho}_{\E,p}\sim \overline{\rho}_{E',p}$.

\begin{lemma}
	If $E'$ does not have a $K$-rational point of order $2$ and is not isogenous to an elliptic curve with a $K$-rational point of order $2$ then $p<C_{E'}$, where $C_{E'}$ is a constant depending only on $E'$.
\end{lemma}
\begin{proof}
	It follows from \cite[Theorem 2]{Katz}. For the proof we refer the reader to \cite[Lemma 3.6]{isik_kara_ozman_2022}.
\end{proof}

\begin{thm}\label{thm:curveexistence}
	Let $K$ be a number field. If $K$ has at least one complex embedding assume Conjectures~\ref{conj:modularity} and~\ref{conj:fakecurves}. Let $\id{P}\in S_K$ some fixed prime. There exists a constant $B_K$ depending only on $K$ such that if $(a,b,c)\in W_\id{P}$ and $p>B_K$ then there is an elliptic curve $E'/K$ such that the following hold:
	\begin{enumerate}[(i)]
		\item $E'$ has good reduction outside $S_K$.
		\item $E'$ has a $K$-rational point of order $2$.
		\item $\overline{\rho}_{\E,p}\sim\overline{\rho}_{E',p}$.
		\item $v_{\id{P}}(j')<0$, where $j'$ is the $j$-invariant of $E'$.
	\end{enumerate}
\end{thm}

\begin{proof}
	If $E'$ does not have a $K$-rational point of order $2$ then it is isogenous to an elliptic curve $E''$ with a $K$-rational point of order $2$ (if $p>C'_E$), so $\overline{\rho}_{\E,p}\sim\overline{\rho}_{E'',p}$. In such a case, we can replace $E'$ by $E''$ and find an elliptic curve satisfying (i), (ii) and (iii). It remains to prove that $v_{\id{P}}(j')<0$. By Lemma~\ref{lemma:2ismultiplicative}, $p\mid \# \rho_{E',p}(I_{\id{P}})$ (since $\# \overline{\rho}_{E',p}(I_{\id{P}}) =\# \overline{\rho}_{\E,p}(I_{\id{P}})$). Assuming $p>24$ this gives the result, by \cite[Lemma 5.1]{SS}.
\end{proof}

\subsection{Contradiction and proof of the main results}


In this subsection we will handle the last step of the modular method, giving the proof of Theorem~\ref{thm:main}. Then, we will prove Corollaries~\ref{coro:quadraticase} and~\ref{thm:Q2}, and we will finish with a remark about how to extent the known results for signature $(p,p,2)$. 

According to our strategy, we need to contradict Theorem~\ref{thm:curveexistence}. Here is where we use hypothesis \textbf{(}\textbf{H}$_\textbf{1}$\textbf{)} or \textbf{(}\textbf{H}$_\textbf{2}$\textbf{)}.
 The following proposition follows from the proof of \cite[Theorem 3]{mocanu}. We write a proof in order to clarify.
\begin{prop}[Mocanu]\label{prop:mocanu} Let $K$ be a number field satisfying \textbf{(}\textbf{H}$_\textbf{2}$\textbf{)} and let $\id{P}\in T_K$. Then, there are no elliptic curves $E/K$ satisfying the following:
	\begin{itemize}
		\item $E$ has good reduction outside $S_K$,
		\item $E$ has a $K$-rational point of order $2$,
		\item $v_{\id{P}}(j)<0$,  where $j$ is the $j$-invariant of $E$.
	\end{itemize}
\end{prop}
\begin{proof}
	Suppose there is an elliptic curve $E$ satisfying (i)--(iii). Since it has a $K$-rational point of order $2$ then we can assume that is of the form 
	\[E:y^2=x^2+ax^2+bx,\]
	with $j$-invariant $j=\frac{2^8(a^2-3b)^3}{b^2(a^2-4b)}$. Since $E$ has good reduction outside $S_K$ then $v_\id{q}(j)\ge0$ for all $\id{q}\not\in S_K$, therefore $j\in\Om_{S_K}$. Consider $\lambda:=\frac{a^2}{b}$ and $\mu:=\lambda-4=\frac{a^2-4b}{b}$. By \cite[Lemma 17 (i)]{mocanu} we get that
	\[ \lambda\Om_K=I^2J, \quad \text{ where } J \text{ is an } S_K\text{-ideal.}\]
	Then, as elements of $\Cl(K)$,  $[J]^{-1}=[I]^{2}$ and also $[J]\in\langle S_K \rangle$. This implies that $[I]\in\Cl_{S_K}(K)[2]=\{1\}$, so $[I]\in \langle S_K \rangle$. In other words, there exist $\gamma\in\Om_K$ and an $S_K$-ideal $\tilde{I}$ such that $I=\gamma\tilde{I}$. Then,
	\[\lambda\Om_K=\gamma^2\tilde{I}^2J, \quad \text{ where both } \tilde{I} \text{ and } J \text{ are } S_K\text{-ideals}.\]
Finally, $\frac{\lambda}{\gamma^2}\Om_K$ is an $S_K$-ideal, which implies that $u := \frac{\lambda}{\gamma^2}$ is an $S_K$-unit. Now, dividing $\mu + 4 = \lambda$ by $u$, we get
\begin{equation}\label{eq:mocanu}
	\alpha+\beta=\gamma^2,
\end{equation}
where $\alpha :=\frac{\mu}{u}\in\Om_{S_K}^\times$, $\beta := \frac{4}{u}\in \Om_{S_K}^\times$. In order to get a contradiction, it is enough to see that $v_\id{P}(j)\ge 0$. Using that $\frac{\alpha}{\beta}=\frac{\mu}{4}$ and equation~(\ref{eq:mocanu}), we can rewrite condition $|v_\id{P}(\frac{\alpha}{\beta})|\le 6v_\id{P}(2)$ as
\[-4v_\id{P}(2)\le v_\id{P}(\mu)\le 8v_\id{P}(2).\]
Now, the proof follows identically as in \cite[Theorem 3]{mocanu}, analizing the possibles values of $v_\id{P}(\mu)$, and using that
\[v_\id{P}(j)=8v_\id{P}(2)+3v_\id{P}(\mu+1)-v_\id{P}(\mu).\]
\end{proof}
Regarding \textbf{(}\textbf{H}$_\textbf{1}$\textbf{)}, we will make use of the following Theorem of Freitas, Kraus and Siksek (with $\ell=2$).
\begin{thm}[\cite{FKS}, Theorem 1]\label{thm:FKS}
	Let $\ell$ be a rational prime. Let $K$ be a number field satisfying the following conditions:
	\begin{itemize}
		\item $\Q(\zeta_\ell)\subset K$, where $\zeta_\ell$ is a primitive $\ell$-th root of unity;
		\item $K$ has a unique prime $\lambda$ above $\ell$;
		\item $\gcd(h_K^+,\ell(\ell-1))=1$.
	\end{itemize}
		Then there are no elliptic curves $E/K$ with a $K$-rational $\ell$-isogeny, good reduction
away from $\lambda$ and potentially multiplicative reduction at $\lambda$.
\end{thm}

Now we are in conditions to prove the main theorem and its implications.
\begin{proof}[Proof of Theorem~\ref{thm:main}]
	Let $K$ be a number field satisfying \textbf{(}\textbf{H}$_\textbf{1}$\textbf{)} or \textbf{(}\textbf{H}$_\textbf{2}$\textbf{)} and let $\id{P}$ in $S_K$ in the former case and $\id{P}$ in $T_K$ in the latest. Take $B_K$ as in Theorem~\ref{thm:curveexistence}. Suppose  there is a non-trivial primitive solution $(a,b,c)$ of~(\ref{eq:42p}) with exponent $p>B_K$ such that $\id{P}\mid c$. By Remark~\ref{rem:valuation2}, we can assume $(a,b,c)\in W_\id{P}$, otherwise $(a,-b,c)$ would be in $W_\id{P}$. By Theorem~\ref{thm:curveexistence} we can find an elliptic curve $E'$ having a $K$-rational point of order $2$, good reduction outside $S_K$, multiplicative reduction at $\id{P}$ and related by $\overline{\rho}_{\E,p}\sim\overline{\rho}_{E',p}$. If \textbf{(}\textbf{H}$_\textbf{1}$\textbf{)} is satisfied, Theorem~\ref{thm:FKS} gives a contradiction. In case that \textbf{(}\textbf{H}$_\textbf{2}$\textbf{)} is satisfied, then Proposition~\ref{prop:mocanu} gives the desired contradiction.
\end{proof}

\begin{proof}[Proof of Corollary~\ref{coro:quadraticase}]
	Note that if $K=\Q(\sqrt{d})$ with $d>0$ a prime number such that $d\equiv5\pmod 8$ then $2$ is inert in $K$. Since $d$ is prime, by \cite[Proposition 1.3.2]{Lemmermeyer} we have $2\nmid h_K^+$  so \textbf{(}\textbf{H}$_\textbf{1}$\textbf{)} holds . Moreover, the effectivity of the bound follows from modularity of curves over real quadratic fields (see Theorem~\ref{thm:FHS}).
\end{proof}

\begin{remark}
	It is easy to see that  Corollary~\ref{coro:quadraticase} still holds for $d<0$ if we assume Conjectures~\ref{conj:modularity} and~\ref{conj:fakecurves}.
\end{remark}

\begin{proof}[Proof of Corollary~\ref{thm:Q2}]
	Since $\Q_{r,2}$ is a totally real field with odd narrow class number (see \cite[II]{Iwasawa}) then \textbf{(}\textbf{H}$_\textbf{1}$\textbf{)} is satisfied and we can apply Theorem~\ref{thm:main}. The  effectivity of the bound follows from modularity of curves over $\Z_p$-extensions of $\Q$ proved by Thorne \cite{Thorne}.
\end{proof}

\begin{remark}
	As was mentioned in Remark~\ref{rem:signature}, the same strategy applies for others signatures. In the particular case of signature $(p,p,2)$, we also have $\ell=2$ (following the remark notation). In \cite{mocanu} just \textbf{(}\textbf{H}$_\textbf{2}$\textbf{)} was used as a possible hypothesis, but in that case Theorem~\ref{thm:FKS} could also be applied in the last step of the strategy, so \textbf{(}\textbf{H}$_\textbf{1}$\textbf{)} would also work.  Note that just assuming \textbf{(}\textbf{H}$_\textbf{1}$\textbf{)}, Theorems 5 and 7 of \cite{mocanu} can be generalized. In the general case, we can get the following result.
	
	\begin{thm}
		Let $K$ be a number field satisfying \textbf{(}\textbf{H}$_\textbf{1}$\textbf{)} or \textbf{(}\textbf{H}$_\textbf{2}$\textbf{)}. If $K$ has at least one complex embedding, assume Conjectures~\ref{conj:modularity} and~\ref{conj:fakecurves}. Let $\id{P}$ the only prime in $S_K$ in case that \textbf{(}\textbf{H}$_\textbf{1}$\textbf{)} holds or be in $T_K$ in case that \textbf{(}\textbf{H}$_\textbf{2}$\textbf{)} holds. Then, there exists a bound $B_{K}$ (depending only on $K$) such that if $p>B_K$  the equation $x^p+y^p=z^2$ does not have non-trivial primitive  solutions $(a,b,c)\in\Om_K^3$ with $\id{P}\mid b$.
	\end{thm}
	
\end{remark}

\bibliographystyle{alpha}
\bibliography{biblio}

\end{document}